\documentclass[a4,11pt]{article}
\usepackage[a4paper,body={5.1 in,8.1in}]{geometry} 
\usepackage{epsfig}
\usepackage{amsmath}
\usepackage{amssymb}
\usepackage{amsthm}
\usepackage{epsfig}
\usepackage{verbatim}
\usepackage{color}
\include{pstricks}

\include{pspicture}

\newcommand{\ninseps}[3]{
\begin{figure}[h]
\begin{center}
 \scalebox{#3}{\includegraphics{#1}}
\end{center}
\caption{\hspace{0.25cm}#2\label{f:#1}}
\end{figure}
}

\DeclareSymbolFont{AMSb}{U}{msb}{m}{n}
\DeclareMathSymbol{\N}{\mathbin}{AMSb}{"4E}
\DeclareMathSymbol{\Z}{\mathbin}{AMSb}{"5A}
\DeclareMathSymbol{\R}{\mathbin}{AMSb}{"52}
\DeclareMathSymbol{\Q}{\mathbin}{AMSb}{"51}
\DeclareMathSymbol{\I}{\mathbin}{AMSb}{"49}
\DeclareMathSymbol{\C}{\mathbin}{AMSb}{"43}
\DeclareMathSymbol{\F}{\mathbin}{AMSb}{"46}
\DeclareMathSymbol{\LL}{\mathbin}{AMSb}{"4C}

\newtheorem{theorem}{Theorem}[section]
\newtheorem{lemma}[theorem]{Lemma}

\newtheorem{definition}{Definition}[section]
\newtheorem{remark}{Remark}

\newtheorem{assumption}{Assumption}

\usepackage{hyperref}
\begin{document} 
\title{Asymptotically Optimal Importance Sampling for Jackson Networks with a Tree Topology}

\author{{\normalsize Ali Devin Sezer}\\
\small Institute of Applied Mathematics\\
\small Middle East Technical University\\
\small Ankara, Turkey
\\
\\ 
}


\maketitle
\thispagestyle{empty}
\begin{abstract}
Importance sampling (IS) is a variance reduction method
for simulating rare events. 
A recent paper by Dupuis, Wang and Sezer (Ann. App. Probab. 17(4):1306- 1346, 2007) 
exploits connections between IS and subsolutions to a limit HJB
equation and its boundary conditions
to show how to design and analyze simple and efficient IS
algorithms for various overflow events for tandem Jackson networks.
The present paper uses the same subsolution
approach to build asymptotically optimal IS schemes
for stable open Jackson networks with a tree topology. 
Customers arrive at the single root of the tree.
The rare overflow event we consider is the following: 
given that initially the network is empty, the system
experiences a buffer overflow before returning to the empty state.
Two types of buffer structures are considered: 1) A single system-wide
buffer of size $n$ shared by all nodes, 2) each node $i$ has its
own buffer of size $\beta_i n$, $\beta_i \in (0,1)$.
\end{abstract}
\section{Introduction}
Importance sampling (IS) is a method for simulation of rare events. 
It is used
in many applications including simulation of communication systems, computation of credit 
risk and pricing of financial
derivatives. 
The idea in IS is to change the sampling distribution (and modify the Monte
Carlo estimator accordingly) to reduce estimator variance. 
Queuing processes are basic stochastic models that are commonly used
in a wide range of application areas.
The simplest type of queuing processes are Jackson networks,
in which the arrival and service times at the nodes of the network are 
assumed to be
independent and exponentially distributed with constant rates.

In the present paper we build an IS algorithm, which is optimal in a certain
asymptotic sense (see Section \ref{s:IS}), to simulate buffer overflows of
stable open Jackson networks with a tree topology.
The system
is stable in the sense that the average service rate at each
node is faster than the average arrival rate to that node.
Customers arrive at the single root of the tree. 
The rare overflow event we consider is the following: 
given that initially the network is empty the system
experiences a buffer overflow before returning to the empty state.
Two types of buffer structures are considered: 1) A single system-wide
buffer of size $n$ shared by all nodes 2) each node $i$ has its
own buffer of size $\beta_i n$, $\beta_i \in (0,1)$.

To construct our optimal IS algorithms we use an optimality 
result from \cite{thesis}
which was obtained using the optimal control/subsolution approach to IS of 
\cite{DSW, duphui-is1,duphui-is2,duphui-is3,duphui-is4}.
This result states that
to construct optimal IS algorithms for the simulation of a wide range
of buffer overflow events of any stable Jackson network
it is sufficient to build appropriate smooth
subsolutions to a Hamilton Jacobi Bellman (HJB) equation  and its boundary conditions
(these are given in \eqref{e:DPE} in the context we study in the current paper). 
This HJB equation
and the boundary conditions 
are the main tools of the optimal control/subsolution approach and are derived
from an optimal control representation of the IS distribution construction problem.

The main contribution
of the present paper is a recursive algorithm which takes as input the parameters of an 
arbitrary Jackson network with a tree topology and constructs a smooth subsolution to 
the HJB equation and its boundary conditions given in \eqref{e:DPE}.
The constructed subsolution is of the form of a smoothed minimum of affine functions, as was
the case in previous works using the subsolution approach, e.g. \cite{DSW,thesis}.
The quantities
that appear in the subsolution 
(and hence the algorithm) have simple heuristic interpretations as 
{\em effective } utilities and rates of nodes
in the system. They are ``effective'' in the sense
that they depend on whether a node is empty or nonempty. 
These concepts are explained in detail in subsection \ref{ss:ssubsol}.
The main results of the paper are Lemmas \ref{l:main} and \ref{l:rectangle} which
prove that 
the subsolutions arising from the effective rates and utilities
satisfy all the conditions of 
the general optimality theorem in \cite{thesis} for both type of buffer 
structures that we will be studying in this paper. Numerical results
in Sections \ref{s:numerical} and \ref{s:rectangle} demonstrate the
practical usefulness of the resulting IS algorithms.

Since the initial writing
\cite{istrees} 
of the present paper 
a recent paper by Dupuis and Wang 
\cite{yeniDW} appeared
that treat the IS problem for any stable Jackson network
using the subsolution approach.
The relation between the results in the current paper and those
in \cite{yeniDW} is discussed in Section \ref{s:discussion}.

There is a tremendous amount
of work on the IS of queueing networks, which include,
\cite{Rubetal04,WeiQui,Sadowsky91, Changetal, KroeseNicola,JunejaNicola,ParWal,
GlassKou,
BoerNicola02}.
The problem of constructing IS algorithms for buffer overflow of queuing networks 
was first posed for the simple two node tandem network in \cite{ParWal},
which also proved that a static large deviations based change of measure
is asymptotically optimal for certain parameter values of the system.
An asymptotically optimal IS algorithm with
optimality proofs for buffer overflow of stable tandem Jackson networks was first developed in \cite{DSW} using the
optimal control/subsolution approach. The discontinuous dynamics of the queuing process near
the boundaries of its state space (i.e., when few customers remain
in some of the nodes) makes the IS construction problem for queuing networks 
difficult \cite{DSW,GlassKou}.
This property rules out iid sampling distributions
(such as those developed in \cite{Siegmund} in the context of a random walk on the 
real line and in \cite{ParWal} in the context of two tandem Jackson nodes) as candidates for 
efficient IS samplers and forces one to search for a good IS distribution among dynamic 
distributions, where indeed the subsolution approach locates the optimal IS distributions. 
For a more in depth discussion of these issues we refer the
reader to \cite{DSW,thesis,GlassKou, duphui-is1}.

\section{Setup}\label{s:setup}

We consider Jackson networks with a tree topology. 
Customers
arrive only at the root of the tree. 
Our goal is to construct optimal IS algorithms to estimate the following probability:
\begin{equation}\label{e:ep1}
P_0( \text{system experiences an overflow before it empties}).
\end{equation}
This overflow event depends on the buffer structure of the 
network, which
will be made precise in subsection \ref{ss:overflow}.
For the computation of $p_0$ it is enough to consider the
embedded discrete time random walk of the Jackson network. The normalized
service and arrival rates and the routing probabilities of the Jackson network
are the jump probabilities of the embedded random walk.

\subsection{Notation and Definitions} \label{ss:notation}

The tree consists of $d$ nodes. 
$X(i)$ is the population of $i^{th}$ node at the jump times in the network.
$i\rightarrow j$ denotes that node $j$ is a child of node $i$.
For $i \rightarrow j$,
$\mu_{i,j}> 0 $ is the rate at which customers are served in node
$i$ and are either [sent to node $j$ if $j>0$] or [ leave the system
if $j=0$]. 

Total service rate at node $i$ is defined as
$\mu_i \doteq \sum_{k} \mu_{i,k}.$
Arrival rate to node $\Lambda_j$ at node $j$ equals
$\lambda$ if $j$ is the root node. Otherwise it equals
$\Lambda_j \doteq \Lambda_i \frac{\mu_{i,j}}{\mu_i}$
where node $i$ is the parent of node $j$.
It is no loss of generality to assume that
$\lambda + \sum_{i=1}^d \mu_i$ equals $1$ ; otherwise
one can change the time unit so that the equality holds.
The utility of node $i$ is defined
as:
$\rho_i \doteq \Lambda_i / \mu_i.$
The Jackson network is called stable if $\rho_i < 1$ for all
$i \in\{1,2,...,d\}$.
Therefore we assume
that
$\vee_{i=1}^d \rho_i  < 1.$
This stability assumption implies that the buffer overflow events of
interest we study in the present paper decay exponentially in $n$ (see \eqref{e:ldres} and
\eqref{e:rectrate}). 
Asymptotic optimality of an IS algorithm is stated in terms of this 
exponential decay (see Section \ref{s:IS}).

The evolution of the random walk $X$ takes place in the state space ${\mathbb Z}^d_+$. 
This set has $2^d - 2$ different boundaries:
$\partial_i \doteq 
\{x =(x_1,x_2,...,x_d) \in{\mathbb Z}^d_+: x_i = 0 \}$, $i \in \{1,2,...,d\}$,
$\partial_{\{i_1,i_2,...,i_k\}} \doteq \bigcap_{l=1}^k \partial_{i_l}$,
$\{i_1,i_2,...,i_k\}$ $\subset \{1,2,...,d\}.$
As we have remarked earlier the dynamics of $X$ depends on whether $X$ is  on one
of these boundaries and if so it further depends on which one. 
We will find it convenient to identify these boundaries with 
bitmaps $b\in \{0,1\}^d$.
$b$ describes the following state of the network: $b(i)=0$ signifies that node $i$ is empty,
$b(i)=1$ signifies that it is non-empty.
Define
$
v_{0,1} = (1,0,...,0)
$
and
\begin{align*}
{\mathcal V}_2 &\doteq \left\{ v_{i,j}, i,j \in \{1,2,...,d\}: i\rightarrow j, \right.\\
&~~~~~~~~~~~~~~~ \left. v_{i,j}(i)=-1,~v_{i,j}(j)=1,~v_{i,j}(k)=0, k \in \{1,2,...,d\}- \{i,j\} \right\}\\
{\mathcal V}_3 &\doteq \left\{ v_{i,0}, i \in \{1,2,...,d\}:
  v_{i,0}(i)=-1,~v_{i,0}(k)=0, k \in \{1,2,...,d\}- \{i\} \right\}
\end{align*}
Let ${\mathcal V} \doteq \{ v_{0,1} \} \cup {\mathcal V}_2 \cup {\mathcal V}_3.$
${\mathcal V}$ are the set of all possible jumps the process $X$ can make.
$v_{0,1}$ corresponds to a new customer arriving at the root node,
$v_{i,j} \in {\mathcal V}_2$
corresponds to server $i$ serving a customer in queue $i$ and
sending it to queue $j$ with $i\rightarrow j$, and finally
$v_{i,0} \in {\mathcal V}_3$ corresponds to a customer leaving the system after being served
by  server $i$.

Let $Y=\{Y_k: k =0,1,2,...\}$ be an iid
sequence such that
$P_x( Y_k = v_{0,1} ) =p(v_{0,1}) \doteq \lambda$,
$P_x( Y_k = v_{i,j} ) = p(v_{i,j}) \doteq \mu_{i,j}$ for $v_{i,j} \in {\mathcal V}_2$,
$P_x( Y_k = v_{i,0} ) = p(v_{i,0}) \doteq \mu_{i,0}$ for $v_{i,0}  \in {\mathcal V}_3$,
for all $x \in {\mathbb Z}^d_+$.
$Y_k$ are the unconstrained increments of the process $X$.
We assume the existence of a probability space
$(\Omega, {\mathcal F})$ equipped with the probability distributions $P_x$. The
subscript $x$ denotes the initial position of the queuing system $X_0$: under $P_x$,
$X_0=x$ almost surely.

$X \in \partial_{\{i_1,i_2,...,i_k\}}$  if the Jackson
network has no customers in queues $i_1$, $i_2$,...,and $i_k$. Therefore
 $v_{l,j}$, $j \in \{0,1,2,..,d\}$, $l \in \{i_1,i_2,...,i_k\}$,
cannot be an increment of $X$  when $X \in \partial_{\{i_1,i_2,...,i_k\}}$.
The constraining map $\pi:{\mathbb R}^d_+ \times {\mathcal V} \rightarrow 
{\mathcal V} \cup \{0\}$ will make
sure that this does not happen:
$$
\pi(x,v) =
\begin{cases}
0, &\text{if }
x\in \partial_i \text{ for some }i\in\{1,2,...,d\} \text{ and }
 \langle v, n_i \rangle < 0,\\
v,
 &\text{otherwise, } 
\end{cases}
$$
where $n_i$ is normal to the boundary $\partial_i$: $n_i(i)= 1$ and $n_i (j) =0$ for $j\neq i$.
$X$  can now be written as
\begin{equation}\label{e:dynamics}
X_{k+1} \doteq X_k + \pi(X_k,Y_k).
\end{equation}
$X_0$ is the initial state of the system and under $P_x$ it equals $x\in {\mathbb Z}^d_+$
almost surely.
\subsection{Overflow event of interest}\label{ss:overflow}
We would like to develop IS algorithms to estimate \eqref{e:ep1}. 
We now define what we mean by an overflow.
Let  $\partial_+^d \doteq \{ x \in {\mathbb R}^d_+: \vee_i x(i) = 1 \}$.
\begin{assumption}\label{as:bs}
The system has a buffer whose
structure is determined
by a normalized exit set ${\mathcal S} \subset [0,1]^d$ with the following properties:
1) ${\mathcal S}$ is closed and connected, 
2) $0 \notin {\mathcal S}$,
3) Any continuous curve in $[0,1]^d$ that contains $0$ and
a point from $\partial_+^d$ must also contain a point from ${\mathcal S}$.
4) For $S_n \doteq \{x \in {\mathbb Z}^d_+: x/n \in {\mathcal S} \}$,
\begin{equation}\label{e:eda}
\gamma \doteq \lim_{n \rightarrow \infty} - \frac{1}{n}\log P_{\bf s}( X \text{ hits }  S_n
\text{ before } 0)
\end{equation}
exists and is nonzero. 
\end{assumption}

In this article we are interested in two types of buffer structures:
1)
$
{\mathcal S}_1 \doteq \{x \in {\mathbb R}^d_+ : x(1) + x(2) + \cdots + x(d) = 1\}.
$
$S_n \doteq \{x \in {\mathbb Z}^d_+: x/n \in {\mathcal S}_1 \}$ corresponds to a single buffer of size $n$
shared by all queues.
For $\beta \in {\mathbb R}^d_+$ 
$
{\mathcal S}_2 = \{x \in {\mathbb R}^d_+ : x(i) = \beta(i) \text{ for some $i$ and } x(j) \le \beta(j) \text{ for all }j   \}.
$
Then
$S_n \doteq \{x \in {\mathbb Z}^d_+: x/n \in {\mathcal S}_2 \}$ corresponds to $d$ independent buffers,
one for each node. The size of the buffer for node $i$ is given by $n\beta(i)$.
Without loss of generality we will assume that $\vee_i \beta(i) = 1.$

Define
the initial point
${\mathbf s} \doteq (1,0,0,0,\dots,0)$.
Fix a buffer structure ${\mathcal S}$ and define the exit boundaries $S_n$ as above.
We now rewrite the exit probability of interest precisely as:
$
p_n \doteq P_{\bf s}( X \text{ hits } S_n \text{ before it hits } 0 ).
$
We consider the case ${\mathcal S} = {\mathcal S}_1$ (all nodes share a single buffer)
in Section \ref{s:shared}
and the case ${\mathcal S} = {\mathcal S}_2$ (one buffer for each node) in Section \ref{s:rectangle}.

\section{Importance Sampling}\label{s:IS}
In order to simulate  $X$ using importance sampling one specifies
a sampling distribution $\bar{p}(v|x)$, $v \in {\mathcal V}$ and 
$x \in {\mathbb Z}^+$ and simulates $X$ from this distribution.
Note that we allow $\bar{p}$ to depend on $x$, the current position of $X$.
Define $A_n$ to be the set of sample paths that hit the exit set $S_n$ before $0$
and let $T_n$ denote the first time $X$ hits $S_n$ or $0$.
The IS estimator of $p_n$ using $K$ sample paths is then:
\begin{equation}\label{e:ISestimator}
	\frac{1}{K} \sum_{k=1}^K \hat{p}_n^k,
	~~~~~~ \hat{p}_n^k \doteq  1_{A_n}(X^k) \cdot \prod_{i=1}^{T_n -1} \frac{p(Y^k_i)}{\bar{p}(Y^k_i|X^k_i)},
\end{equation}
where 
$X^k$ denotes the $k^{th}$ independent sample path used in the simulation.
The increments $\{Y^k\}$ 
are iid copies of the increment process $Y$ sampled from $\bar{p}$. $X^k$ is 
built along with $Y^k$ using the dynamics \eqref{e:dynamics}.
The product is the likelihood ratio of $P_{\bf s}$ and $\bar{P}$, which
appears in the estimator to cancel off the effect of changing the sampling distribution
from $p$ to $\bar{p}$.

$\hat{p}_n\doteq \hat{p}_n^1$ is an unbiased estimator of $p_n$
and therefore the variance of $\hat{p}_n$ depends on the sampling
distribution only through the second moment of $\hat{p}_n$.
Because $p_n$ decays exponentially, one would like the second moment of
$\hat{p}_n$ to decay exponentially as well.
However, Jensen's inequality implies that
$$
\limsup_n -\frac{1}{n} 
\log \hat{\mathbb E}[\hat p_n^2] \le 
\limsup_n -\frac{2}{n}  \log \hat{\mathbb E}[\hat p_n]\equiv 2\gamma.
$$
In other words, the exponential decay rate of the second moment can be at
most twice that of the probability.
The IS estimator is said to be {\it asymptotically optimal} if
the upper bound is achieved, i.e., if
$\liminf_n -\frac{1}{n}\log {\mathbb E}[\hat p_n^2] \ge 2\gamma.$

\subsection{Definitions from the subsolution approach}
In this subsection we will give only the definitions from the subsolution
approach 
that we need to present the results and the algorithm for the tree Jackson 
networks. A full development of the subsolution approach ideas can 
be found in \cite{duphui-is3, duphui-is4,DSW}.

\paragraph{Hamiltonians, the limit HJB equation and the boundary conditions.}
For a bitmap $b\in\{0,1\}^d$ and $q \in {\mathbb R}^d$ define
\begin{align}\label{e:Hams}
N_b(q) &\doteq 
\lambda e^{-q(1)/2}+
\sum_{i:b(i)=1}
\sum_{ i\rightarrow j} \mu_{i,j} e^{\frac{q(i) - q(j)}{2}}\notag
+
\sum_{i:b(i) = 1} \mu_{i,0} e^{\frac{q(i)}{2}}
+
\sum_{i:b(i) = 0 } \mu_i,\\
H_b(q) &= -2\log N_b(q).
\end{align}
$H_b$ is the Hamiltonian associated with boundary $b$. We denote $H_b$
by $H$ if $b=(1,1,1,\dots,1,1)$.

For $x \in {\mathbb R}^d_+$, define $b_x\in\{0,1\}^d$ as follows:
\begin{equation}\label{e:bx}
b_x(i) \doteq \begin{cases}
0,& \text{ if } x(i) = 0,\\
1,& \text{ otherwise.}
\end{cases}
\end{equation}
$b_x$ indicates which boundary $x$ is on (if $b_x=(1,1,\dots,1,1)$ then
$x$ is in the interior of ${\mathbb R}^d_+$).

\paragraph{Definition of a subsolution.}
The limit HJB equation
and its boundary conditions
that are in the center of the subsolution approach are as follows:
\begin{equation}\label{e:DPE}
H(DV(x))=0,~~
H_{b_x}(DV(x)) = 0,
\end{equation}
where $DV$ denotes the gradient of $V$.
A subsolution to \eqref{e:DPE} is defined as follows:
\begin{definition}\label{d:smoothsubsolg}
$\bar{V}$ is an $\epsilon$-subsolution to \eqref{e:DPE}
if it is $C^1({\mathbb R}^d,{\mathbb R})$ and
\begin{enumerate}
	\item [(a)]
$~H_{b_x}( D\bar{V}(x)) \geq -\epsilon$ for all  $x \in {\mathbb R}^d_+$,
\item[(b)]
 $~\bar{V}(0) \ge 2\gamma -\epsilon$,
 \item[(c)]
$~\bar{V}(x)\leq \epsilon ,x\in {\mathcal S},$
\end{enumerate}
where $\gamma$ is the decay rate associated with the buffer structure ${\mathcal S}$.
\end{definition}
For $q \in {\mathbb R}^d$
and bitmap $b$ define the jump probabilities:
\begin{equation}
\label{e:jump}
\bar{p}^*_b(q)(v_{i,j})
=\begin{cases}
	\lambda \frac{\exp(-q(j)/2)}{N_b(q)},~~~~& i = 0, ~j=1\\
	\mu_{i,j}\frac{\exp( (q(i)-q(j))/2)}{N_b(q)},~~~~& i \neq 0, b(i) = 1, i\rightarrow j\\
	\mu_{i,0} \frac{ \exp(q(i)/2)}{N_b(q)},~~~~& i \neq 0,~~ b(i)= 1 \\
\mu_{i,j}\frac{1}{N_b(q)},~~~~&i \neq 0,~  b(i)= 0, i\rightarrow j \text{ or } j = 0.
\end{cases}
\end{equation}
Any smooth function $W:{\mathbb R}^d \rightarrow {\mathbb R}$ can be used
to define a stochastic kernel $\bar{p}$ as follows:
\begin{equation}\label{e:naive}
\bar{p}_W(v | x) = \bar{p}^*_{b_x}(v| DW(x/n)),
\end{equation}
where $DW$ is the gradient of $W$.

Theorem 4.1.1 of \cite{thesis} asserts that the IS transition kernel defined by 
smooth subsolutions to \eqref{e:DPE} satisfying growth conditions on their
Hessians are asymptotically optimal. For completeness we quote this theorem below.

\begin{theorem}[Theorem 4.1.1 of \cite{thesis}]\label{t:optimality}
Let $\{\bar{V}_n\}$ be a sequence of
$C^2([0,1]^d,{\mathbb R})$ functions that satisfy
1) $\bar{V}_n$ is a $\epsilon_n$-subsolution
2)
$ \left| \frac{\partial^2 \bar{V}_n}{\partial x_i\partial x_j} \right| 
\le \frac{C}{\delta_n}\text{ for }i,j \in \{1,2,...,d\},
$
for some fixed constant $C <\infty$ and
a pair of non negative sequences
$\{ \delta_n\}$ and $\{\epsilon_n\}$ that
converge to $0$ and satisfy
$n\delta_n  \rightarrow \infty$.
Then the IS scheme
defined by the subsolutions 
$\bar{V}_n$
is asymptotically optimal.
\end{theorem}

In the next section
we will construct a sequence of smooth subsolutions to \eqref{e:DPE} that satisfy the 
conditions of this theorem by piecing together at most $2^d$ affine 
functions for the buffer structure ${\mathcal S}_1$. We will find out in Section
\ref{s:rectangle} that the same sequence also works for  ${\mathcal S}_2$ 
(one individual buffer for each node).

\section{Single shared buffer}\label{s:shared}

In this section we will be working with 
${\mathcal S}= {\mathcal S}_1 =  \{x \in {\mathbb R}^d_+ : x(1) + x(2) + \cdots + x(d) = 1\}$.
As noted before, ${\mathcal S}_1$ corresponds to a single buffer shared by all queues in the system.
To remind the reader, we are interested in the overflow probability:
$
p_n \doteq P_{\bf s}( X \text{ hits } S_n \text{ before it hits } 0 ),
$
where 
and
$S_n \doteq \{x \in {\mathbb Z}^d_+: x/n \in {\mathcal S}_1 \}$.
It is proved in  \cite{GlassKou} that 
\begin{equation}\label{e:ldres}
\lim_{n\rightarrow \infty} -\frac{1}{n} \log p_n =  \gamma_1=  \min_{i} -\log \rho_i.
 \end{equation}
In particular, this implies that ${\mathcal S}_1$ satisfies the conditions of
Assumption \ref{as:bs}.

\subsection{The smooth subsolution}\label{ss:ssubsol}
We define the following quantities
to write down the subsolution 
to \eqref{e:DPE}
that
we have in mind.

\paragraph{The effective rate $M_i(b)$ of node $i$ at boundary $b$.}
\begin{equation}\label{e:Mib}
M_i(b) \doteq \begin{cases}
\mu_i, & \text{ if $b(i)=1$},\\
\min\left(\mu_i, 
\sum_{k:i\rightarrow k} 
M_k(b) + \mu'_{i,0}\right), & \text{ if $b(i)=0$},
\end{cases}
\end{equation}
where
$
\mu'_{i,0} \doteq \Lambda_i \frac{\mu_{i,0}}{\mu_i}
$
is the traffic that leaves the system through node $i$.
The recursive formula
\eqref{e:Mib} is the main ingridient of our construction and is suggested by the
definition of the Hamiltonians \eqref{e:Hams} and the HJB equation \eqref{e:DPE} to
which we are constructing a subsolution.
The form of \eqref{e:Mib} 
and the role $M_i(b)$ plays in the solution to the problem
suggests the following interpretation of \eqref{e:Mib}.
\eqref{e:Mib} seems to compute an ``effective'' 
service rate for each node taking into account whether the node is empty
or nonempty.
If a node is nonempty its effective service rate is simply
its service rate.
If the node is empty, \eqref{e:Mib} 
seems to consider it as a system whose components are the 
nodes it directly feeds and computes the effective rate as the total effective rates
of the components. There is also an upper bound on the effective rate, namely the
service rate and if the aforementioned total exceeds this bound then again the effective
rate is set to be the service rate. In this interpretation
$\mu_{i,0}'$ can be thought of as the effective rate of outside of the 
network for the empty node $i$. 

\paragraph{The effective utility $\rho_i(b)$}
$
\doteq  \frac{\Lambda_i}{M_i(b)}.
$
The effective utility of a node is the ratio of its arrival rate to its effective service rate.
If node $i$ is nonempty then it coincides with the ordinary utility $\rho_i$. 

\paragraph{The effective gradient $q\in {\mathbb R}^d$ 
associated with boundary 
$b$.}
\begin{equation}\label{e:qi}
q(i) \doteq 2 \log \rho_i(b) = 2\log \frac{\Lambda_i}{M_i(b)},
\end{equation}
where $q(i)$ denotes the $i^{th}$ component of the vector $q$.
We will use the affine functions defined by the effective gradients to construct
our subsolution of \eqref{e:DPE}. The effective gradient $q$ of the boundary
$b$ will be the gradient of the smooth subsolution around that boundary.

For each boundary $b$ there is an effective gradient $q$. It may happen
that two boundaries $b_1$ and $b_2$ have the same effective gradients.
Let $EG\doteq\{q_1$, $q_2$,...,$q_L\}$, $L \le 2^d $,  be the set of unique effective 
gradients.
We identify two extreme elements of the set $EG$: 
firstly,
the effective gradient corresponding to the boundary $0=(0,0,0,\dots,0,0)$ (all nodes empty)
is $0=(0,0,0,\dots,0,0)$ (this follows from
\eqref{e:Mib} and the definition of $\mu'_{i,0}$). Secondly,
the effective gradient corresponding to the boundary $1=(1,1,1,\dots,1,1)$ (all nodes non-empty)
is the vector whose $i^{th}$ component is $\log \Lambda_i/\mu_i$.

Now define
\begin{equation}\label{e:defmib}
m_i(b) \doteq \begin{cases}
\mu_i, & \text{ if $b(i)=1$},\\
\sum_{k:i\rightarrow k} m_k(b) + \mu'_{i,0}, & \text{ if $b(i)=0$}.
\end{cases}
\end{equation}
The simple gradient $q=(q_1,q_2,...,q_d)$ associated with boundary 
$b$ is defined as 
$
q(i) \doteq 2\log \frac{\Lambda_i}{m_i(b)}
$
where as before $\Lambda_i$ is the arrival rate to node $i$.
The following lemma relates simple and effective gradients.
Bitmaps $b'$ and $b$ satisfy $b' \ge b$ if $b'(i) \ge b(i)$ for all 
$i \in \{1,2,3,...,d\}.$
\begin{lemma}\label{l:effsimp}
Let $q$ be the effective gradient associated with boundary $b$.
Then
there exists a boundary $\bar{b} \ge b$ such that $q$ is the simple
gradient associated with $\bar{b}$.
\end{lemma}
\begin{proof}
If $b=(1,1,1,...,1,1)$ then there is nothing to prove because for this boundary
the effective gradient and the simple gradient are the same.
Then we assume that there are some empty nodes indicated by $b$. $\bar{b}
\ge b$ is constructed as follows. Initially set $\bar{b} = b$. 
For each empty node $i$ in $b$
set $\bar{b}_i$ to $1$ if 
$M_i(b) = \mu_i.$
(see  \eqref{e:Mib}).
It is clear that  1) $\bar{b} \ge b$ and  2)  the 
effective and simple gradients of $\bar{b}$ are the same vector which
is the effective gradient of $b$.
\end{proof}
\begin{definition}\label{d:defa}
For an effective gradient $q_l \in EG$ let $\bar{b}$ be the boundary
whose simple gradient equals $q_l$. Define $\alpha_l$ to be 
the number of $0$'s in $\bar{b}$ plus $1$. 
\end{definition}The $\alpha_l$'s will determine the size
of the regions where the change of measure defined by $q_l$ is used for IS.
Now define the piecewise affine subsolution
\begin{equation}\label{e:daf}
W^{\epsilon}_l(x) = 2\gamma_1  -\alpha_l\epsilon + \langle q_l, x \rangle,~~ W^{\epsilon}(x)= \bigwedge_{l=1}^L W^{\epsilon}_l(x),
\end{equation}
where $L$ is the number of effective gradients and $q_l$ are the effective gradients.
$W^{\epsilon}$ is piecewise affine and not smooth in general.
To obtain the sequence of smooth subsolutions satisfying the assumptions of 
Theorem 4.1.1 of \cite{thesis} one has to let $\epsilon$ depend on $n$ and then 
smooth $W^{\epsilon}$.
One smoothing method that is simple and easy to implement on a computer
is the following \cite{duphui-is3}. Define
\begin{equation}\label{e:Wep}
W^{\epsilon,\delta}(x)  \doteq -\delta 
\log\sum_{l=1}^L \exp\left\{-\frac{1}{\delta} W^{\epsilon}_l(x)\right\}.
\end{equation}
This smoothing algorithm is based on the following fact:
For
$d$ real numbers $a_1$, $a_2$ ,..., $a_d$:
$
-\lim_{\delta\rightarrow 0 } 
\delta \log\left( \sum_{i=1}^d e^{-a_i/\delta }\right) = 
\bigwedge_{i=1}^d a_i.
$
By Lemma 3.12 of \cite{DSW},
$W^{\epsilon,\delta} \rightarrow W^{\epsilon}$ uniformly
as $\epsilon \rightarrow 0$. In addition,
$W^{\epsilon,\delta}$ is continuously differentiable
and a simple direct calculation gives
\begin{equation}\label{e:wi}
DW^{\epsilon,\delta}(x)  = \sum_{l=1}^L w^{\epsilon,\delta}_l(x) q_l,
~~
w_l^{\epsilon,\delta}(x) \doteq 
\frac{\exp\left\{-{W}_l^\epsilon(x)/\delta\right\}}
{\sum_{k=1}^L \exp\left\{- {W}_k^\epsilon(x)/\delta \right\}}.
\end{equation}

\begin{lemma}\label{l:main}
$W^{\epsilon,\delta}$ defined in \eqref{e:Wep} satisfies:
\begin{enumerate}
\item $ H_{b_x}(DW^{\epsilon,\delta}(x)) \ge -C_1 \exp\left(-\frac{\epsilon}{\delta}\right),$
\item $W^{\epsilon,\delta}(0) \ge 2\gamma_1 - \epsilon 
\left(\frac{\delta}{\epsilon}\log\sum_{l=1}^L \exp\left\{\frac{\alpha_l}{\delta/\epsilon}\right\}\right),$
\item $W^{\epsilon,\delta}(x) \le 0$ for $x \in {\mathcal S}_1$,
\item 
$ \left|\frac{\partial^2 W^{\epsilon,\delta}}{\partial x_i \partial x_j }\right|
\le \frac{C_2}{\delta},
$
\end{enumerate}
where $C_1$ and $C_2$ are constants that only depend on the parameters of the network (arrival and service rates
and the routing probabilities).
\end{lemma}
The proof of Lemma \ref{l:main} is in Appendix \ref{a:proof}.
This lemma directly implies
that, for
$\epsilon_n = -\delta_n \log\delta_n$ and $\delta_n$ 
chosen such that $\delta_n \rightarrow 0$
and $n\delta_n \rightarrow \infty$,
the sequence of smooth subsolutions
$W^{\epsilon_n,\delta_n}$ (where $W^{\epsilon,\delta}$ is 
defined as in \eqref{e:Wep})
satisfy the conditions of the optimality
Theorem 4.1.1 \cite{thesis}. This means that the IS scheme defined
by these subsolutions through \eqref{e:naive} is asymptotically
optimal.

Here we repeat an idea from \cite{duphui-is3,DSW}.
The formula \eqref{e:naive} can be used to translate any smooth function into an
IS transition kernel.
However, for the smooth subsolutions there is a slightly different way of defining IS transition
kernels which turn out to be very convenient in computer simulations.

For $x \in {\mathbb Z}_+^d$ define
\begin{equation}\label{e:directav}
\bar{p}^*(v_{i,j}|x) = \sum_{l=1}^L w_l^{\epsilon,\delta}(x/n)\bar{p}_{b_x}^*(q_l)
(v_{i,j}),
\end{equation}
i.e., we switch the order of taking the average against the weights 
$w_l^{\epsilon,\delta}$ and applying the map $\bar{p}^*_{b_x}(\cdot)$ of 
\eqref{e:jump}. 
The advantage of $\bar{p}^*$ of \eqref{e:directav} is that it requires the
computation of $\bar{p}_b^*(q_l)$ only once
at the beginning of the estimation procedure. During the simulation only
the weights are computed dynamically and averages of the precomputed 
$\bar{p}_b^*(q_l)$ will be the IS rates.
Theorem 4.1.1 of \cite{thesis} doesn't cover this way of computing the IS
rates. However, the modification of this theorem to accommodate direct
averaging entails no significant changes.
In the next section we report on the numerical performance of these algorithms.
\subsection{Interpretation of the IS algorithm defined by the subsolution}
Let $b$ a boundary and $q$ its effective gradient.
\eqref{e:directav} essentially uses $\bar{p}_b(q)$ as the IS change of measure
when the queueing process is on the boundary $b$ and away from 
the lower dimensional
boundaries contained in $b$.
Looking at \eqref{e:qi} and \eqref{e:jump}
one sees that $\bar{p}_b(q)$ is simply the following change of measure:
\begin{equation}\label{e:simple}
\bar{\mu}_{i,j} =\begin{cases} \mu_{i,j}, ~~&\text{if node $i$ is empty},\\
	\mu_{i,j}
	\frac{\rho_i(b)}{\rho_j(b)}
	, ~~&\text{if node $i$ is nonempty},
\end{cases}
\end{equation}
where $\rho_i(b)$ and $\rho_j(b)$ are the effective utilities of nodes $i$ and $j$.
These new rates are renormalized so that they sum to $1$.
By convention $\rho_0(b) = 1$, i.e., the outside of the system is thought of as a node with
utility $1$. 
The IS scheme given by \eqref{e:directav} uses a convex combination of \eqref{e:simple}
when the simulated queuing process transitions from one boundary to another.

\eqref{e:simple} illustrates well how the IS change of measure given 
by the subsolution approach works. In the course of a simulation,
the IS change of measure depends on
which nodes are currently empty and nonempty.
The service probabilities of empty nodes are not modified. 
The service probability $\mu_{i,j}$ of a
nonempty node $i$ is modified through a comparison of the traffic 
at the source $i$ and the target $j$;
the service rate is increased if the source is busier,
decreased otherwise. 
The goal seems
to be
to direct traffic to the less strained node.
The traffic is measured by the effective utilities.
For an empty node the effective utility is a value that takes into account
the traffic in the nodes that follow it immediately.
We also note that the arrival rate $\lambda$ is replaced by
$\bar{\lambda} = \lambda\frac{1}{\rho_1(b)}$ which is always larger
than $\lambda$. Therefore the rate of traffic from outside is always increased. Similarly,
the rate of traffic to outside is always decreased.

We would like to also note that the standard state independent heuristic IS algorithms based on 
large deviations results can be thought of as variants of \eqref{e:simple} in which the
standard utilities are used instead of the effective utilities.

\section{Numerical Results}\label{s:numerical}
\paragraph{Choice of $\epsilon$ and $\delta$.}
The IS algorithm defined by $W^{\epsilon,\delta}$ of \eqref{e:Wep} has two parameters $\epsilon$
and $\delta$.
The optimality Theorem \ref{t:optimality} suggest
$\delta \approx C/n$ and $\epsilon \approx -\delta \log \delta$. Asymptotic optimality
criterion is not precise enough to impose a value for $C$. For the choice of this constant
we used experimental evidence.

Once $\epsilon$ and $\delta$ are fixed,  
$\bar{p}^*(v|x)$ of \eqref{e:directav} is used as the IS change of measure.
The effective gradients $q_1,q_2,...,q_L$ and their $\alpha_l$'s
are computed by iterating over all boundaries  $b$ and 
computing the effective gradient of
each of them using the formulas \eqref{e:Mib} and \eqref{e:qi} and the Definition
\ref{d:defa}.

In the following subsections we present simulation results for various
Jackson networks with a tree topology. In all the estimations $K=10000$ sample paths
were used.

\paragraph{Example 1.}
We first consider the network in Figure \ref{f:n11}.
\ninseps{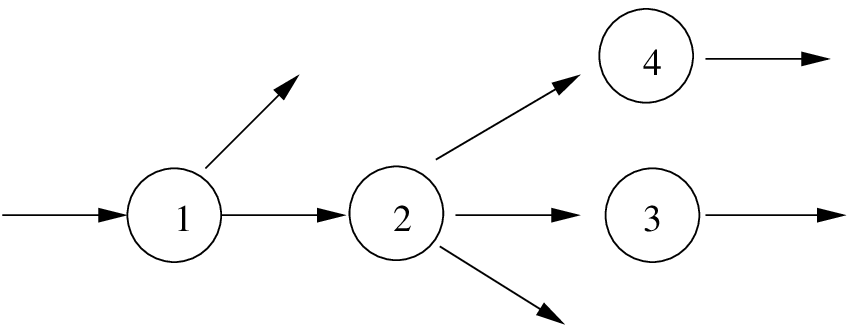}{Example 1}{0.6}
Let us consider the case when
$\lambda = 0.04, \mu_{1,2} =\mu_{1,0}= 0.12,$
$\mu_{2,0} = \mu_{2,3} = \mu_{2,4} = 0.08,$
$\mu_{3,0} = \mu_{3,1} = \mu_{4,0} = \mu_{4,1}= 0.12.$
The node utilities in this case are:
$ \rho_1 = 1/6$, $\rho_2 = 1/12$, $\rho_3 =\rho_4 = 1/36.$
In this example, the utilities are unevenly distributed and
node 1 is the most strained node.
We take $n=30$. 
For $n=30$, and with this four dimensional system, it is possible
to compute
$p_{30}$ without any simulation using the Markov property and straight-forward iteration. 
Such a computation yields $p_{30}=3.269 \times 10^{-23}$.
For the subsolution based IS algorithm we take $\epsilon=0.25$ and $\delta=0.08$.
There turns out to be only five effective gradients for the
given rate values above. 
\begin{table}
\begin{center}
{\tiny Exact probability $p_{30} = 3.269 \times 10^{-23}$

\begin{tabular}{|l|c|c|c|}
\hline & Estimate $\hat{p}_n$ & Standard Error & 95 \% CI\\
 \hline
Est. 1 & $3.50 \times 10^{-23}$& $0.19 \times 10^{-23}$& $[ 3.12,3.88] \times 10^{-23}$\\
\hline 
Est. 2 & $3.22 \times 10^{-23}$& $0.16 \times 10^{-23}$& $[ 2.89,3.54] \times 10^{-23}$\\
\hline 
Est. 3 & $3.28 \times 10^{-23}$& $0.17 \times 10^{-23}$& $[ 2.94,3.61] \times 10^{-23}$\\
\hline 
Est. 4 & $3.32 \times 10^{-23}$& $0.17 \times 10^{-23}$& $[ 2.98,3.66] \times 10^{-23}$\\
\hline 
Est. 5 & $3.16 \times 10^{-23}$& $0.16 \times 10^{-23}$& $[ 2.84,3.48] \times 10^{-23}$\\
\hline 
\end{tabular}
}
\end{center}

\vspace{-0.5cm}
\caption{Simulation Results for Example 1}
\label{t:sim2}
\vspace{-0.5cm}
\end{table}
The results of five consecutive estimations
using the subsolution based IS algorithm are displayed in Table \ref{t:sim2}.
The `standard error' column is the standard error of each estimation. The
$95\%$ confidence intervals are $\hat{p}^n+[-2SE, 2SE]$, where $SE$ is the
standard error displayed under the standard error column . These intervals
are only formal, i.e., we make no assertion about the normality of these errors.
Note that the estimation results are very close to the exact value and the ``$95\%$
confidence intervals'' are accurate: in all these estimations the exact
value happened to be in the computed confidence interval.
In total all five estimations took around 20 seconds on an ordinary laptop 
manufactured in 2004.

\paragraph{Example 2.}
Now we look at the 8-node network depicted in Figure \ref{f:n31}. 
\ninseps{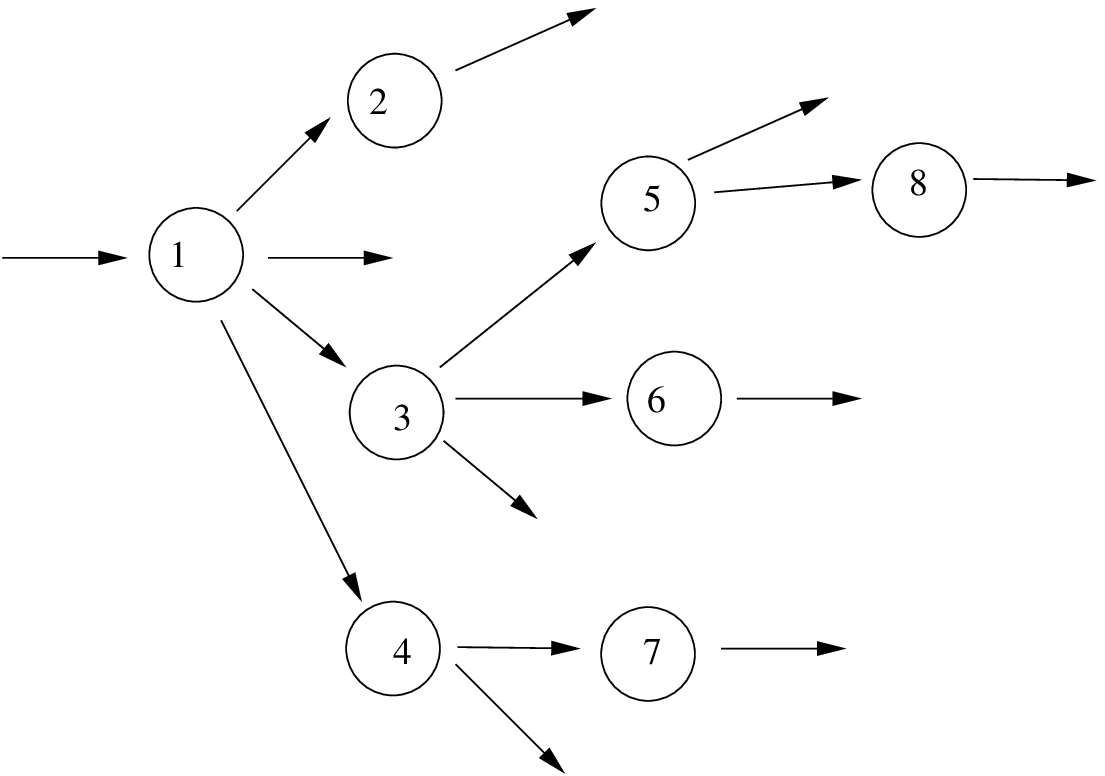}{Example 3}{0.5}
We take the arrival rate $\lambda = 0.1248$, The service rates
are taken to be:
$ \mu_{1,2} = 0.062442 $,
$ \mu_{1,3} = 0.1874$,
$ \mu_{1,4} = 0.062442 $
$ \mu_{1,0} = 0.062517 $
$\mu_{2,0} = 0.06$, 
$\mu_{3,0} = 0.036$, 
$\mu_{3,5} = 0.072$, 
$\mu_{3,6} = 0.072$, 
$\mu_{4,0} = 0.03$,
$\mu_{4,7} = 0.03$,
$\mu_{5,0} = 0.0365$,
$\mu_{5,8} = 0.0365$,
$\mu_{6,0} = 0.073$,
$\mu_{7,0} = 0.025$,
$\mu_{8,0} = 0.028$.
For this choice of the network parameters, the utility of each node turns out to be approximately:
$\rho_1 =  0.331738,$ $\rho_2= 0.3465,$ $\rho_3= 0.3466,$ $\rho_4=0.3465$ , $\rho_5 =0.3419$,
$\rho_6=0.3466,$ $\rho_4=0.3465,$ $\rho_8 = 0.4158.$
All nodes are similarly utilized, although the load on node 8 is slightly heavier then the
rest.
A straightforward simulation with $10^{8}$ samples 
estimate $p_{30}$ to be  $1.2 \times 10^{-6}$ with a standard
error of $1.1\times 10^{-6}$. The subsolution based IS simulation results
are given in Table \ref{t:8n}. The parameters of the algorithm are taken
to be $\epsilon= 0.4$ and $\delta = 0.1$.
Each estimation uses 10000 samples. For this network there are $256$ effective 
gradients. Total run time for all these five estimations was about 20 minutes.

\begin{table}
\begin{center}
{\tiny
\begin{tabular}{|l|c|c|c|}
\hline & Estimate $\hat{p}_n$ & Standard Error & 95 \% CI\\
 \hline
Est. 1 & $1.11 \times 10^{-6}$& $0.17 \times 10^{-6}$& $[ 0.78,1.44] \times 10^{-6}$\\
\hline
Est. 2 & $1.69 \times 10^{-6}$& $0.32 \times 10^{-6}$& $[ 1.04,2.34] \times 10^{-6}$\\
\hline
Est. 3 & $1.25 \times 10^{-6}$& $0.18 \times 10^{-6}$& $[ 0.89,1.61] \times 10^{-6}$\\
\hline
Est. 4 & $1.94 \times 10^{-6}$& $0.51 \times 10^{-6}$& $[ 0.92,2.97] \times 10^{-6}$\\
\hline
Est. 5 & $1.23 \times 10^{-6}$& $0.17 \times 10^{-6}$& $[ 0.89,1.56] \times 10^{-6}$\\
\hline
\end{tabular}

\vspace{-0.2cm}
\caption{Simulation results for the network with eight nodes}
\label{t:8n}
}
\end{center}

\vspace{-0.7cm}

\end{table}

As can be seen, the subsolution based IS algorithm performs very well
for this high dimensional system too: the estimate is within the $95\%$ confidence
interval of the MC estimator and the formal $95\%$ confidence inervals
of the IS simulation do not wildly fluctuate.

\section{Individual Buffers for each Node}\label{s:rectangle}
In this section we look at the buffer structure ${\mathcal S}_2$:
for $\beta \in {\mathbb R}^d_+$ 
$$
{\mathcal S}_2 = \{x \in {\mathbb R}^d_+ : x(i) = \beta(i) \text{ for some $i$ and } x(j) \le \beta(j) \text{ for all }j   \}.
$$
As we noted before,
$S_n \doteq \{x \in {\mathbb Z}^d_+: x/n \in {\mathcal S}_2 \}$ corresponds to $d$ independent buffers,
one for each node. The size of the buffer for node $i$ is given by $n\beta(i)$.
Without loss of generality we will assume that $\vee_i \beta(i) = 1.$
We are, as before, interested in:
$
p_n \doteq P_{\bf s}( X \text{ hits } S_n \text{ before it hits } 0 ),
$
where
$
{\mathbf s} = (1,0,0,\dots,0). 
$
One can prove, using arguments similar to those in  \cite{GlassKou} that 
\begin{equation}\label{e:rectrate}
\lim_{n\rightarrow \infty} -\frac{1}{n} \log p_n =  \gamma_2 =  \min_{i} -\beta(i)\log\rho_i,
\end{equation}
where $\rho_i$ are the node utilities.
In particular, this implies that ${\mathcal S}_2$ satisfies the conditions of
Assumption \ref{as:bs}.
Our goal now is to prove that 
the IS algorithm defined by $W^{\epsilon_n,\delta_n}$ is asymptotically optimal
for the buffer structure ${\mathcal S}_2$ as well
(when buffer structure is changed to ${\mathcal S}_2$, $\gamma_1$ in \eqref{e:daf} needs to be replaced with $\gamma_2$).
To prove this, it is enough to prove a version of Lemma \ref{l:main} for ${\mathcal S}_2$.
Note that only item 3 of this lemma depends on ${\mathcal S}$ and therefore we 
only have to prove that the same item holds for ${\mathcal S}_2$, which
is done in the next lemma.
\begin{lemma}\label{l:rectangle}
Define 
$
W^{c,\epsilon}_l(x) \doteq 2\gamma_2 -\alpha_l\epsilon + \langle q_l,x \rangle,
$
where $\alpha_l$ and $q_l$ are defined as in \eqref{e:qi} and Definition \ref{d:defa}
and $\gamma_2$ is the large deviation rate associated with the boundary ${\mathcal S}_2$ \eqref{e:rectrate}.
Define $W^{\epsilon,\delta}$ by the expression \eqref{e:Wep}. Then:
$
W^{\epsilon,\delta}(x) \le 0.
$
for $x \in {\mathcal S_2}$.
\end{lemma}
\begin{proof}
Take any $x\in{\mathcal S}_2$. Then, there is an $i \le d$ such that $x(i)=\beta(i)$. 
Let $q_L$ be the effective gradient of the boundary $1=(1,1,1,\dots,1,1)$.
$$
W(x) = -\delta 
\log\sum_{l=1}^L \exp\left\{-\frac{1}{\delta}(2\gamma_2  -\alpha_l\epsilon + \langle q_l, x \rangle)\right\}
\le   2\gamma_2 + \langle q_L, x \rangle - \alpha_L\epsilon.
$$
By definition, $q_L(i) = 2\log \frac{\mu_i}{\Lambda_i}$ and the rest of the components of $q_L$ are negative.  These facts, 
\eqref{e:rectrate}, $x \in {\mathbb R}^d_+$, and $x(i)=\beta(i)$ imply that the last display
is less than
$ -\alpha_L \epsilon.$
This finishes the proof of this lemma.
\end{proof}

\paragraph{Numerical example}
Consider a network with five nodes with the following service rates:
$\mu_{1,2} = 0.038$, $\mu_{1,3} = 0.057$, $\mu_{1,0} = 0.095$,
$\mu_{2,4} = 0.076$, $\mu_{2,0} = 0.114$, $\mu_{3,5} = 0.095$, 
$\mu_{3,0}= 0.095$, $\mu_{4,0} = 0.19$, $\mu_{5,0} = 0.19$
and $\lambda = 0.1$.
We will suppose that the buffer sizes for the nodes are respectively:
$15$, $15$, $17$, $18$, $19$
Then $n=19$ and
$\beta(1) =\beta(2) = 15/19$, $\beta(3) = 17/19$, $\beta(4) = 18/19$, $\beta(5) = 1.$
The choice of the buffer sizes are rather arbitrary. We chose them relatively
small so that it was possible to compute the buffer overflow probability $p_{19}$
using the Markov property and direct iteration. The exact value of $p_{19}$
turns out to be $p_{19} = 6.8601 \times 10^{-9}$.

The relative node utilities are:
$\beta(1)\rho_1 = 0.208$, $\beta(2)\rho_2 = 0.042$, $\beta(3)\rho_3 = 0.013$,
$\beta(4)\rho_4 = 0.0004$, $\beta(5)\rho_5 = 0.0008.$
Node $1$ is clearly the most strained node and the loads on the rest of the nodes are spread.
Following the same reasoning as in Section \ref{s:numerical} 
we take $\epsilon=0.3$ and $\delta=0.1$.
The IS simulation now proceeds as before. One uses
$\bar{p}(\cdot|x) = \bar{p}^*(\cdot|x)$ given in \eqref{e:directav} for
the IS change of measure.
\begin{table}
\begin{center}
{\tiny
{\small Exact probability $p_{19} = 6.8601 \times 10^{-9}$}

\begin{tabular}{|l|c|c|c|}
\hline & Estimate $\hat{p}_n$ & Standard Error & 95 \% CI \\
\hline
Est. 1 & $7.33 \times 10^{-9}$& $0.42 \times 10^{-9}$& $[ 6.50,8.17] \times 10^{-9}$\\
\hline
Est. 2 & $6.81 \times 10^{-9}$& $0.34 \times 10^{-9}$& $[ 6.12,7.50] \times 10^{-9}$\\
\hline
Est. 3 & $7.30 \times 10^{-9}$& $0.38 \times 10^{-9}$& $[ 6.53,8.06] \times 10^{-9}$\\
\hline
Est. 4 & $7.05 \times 10^{-9}$& $0.39 \times 10^{-9}$& $[ 6.28,7.83] \times 10^{-9}$\\
\hline
Est. 5 & $7.01 \times 10^{-9}$& $0.37 \times 10^{-9}$& $[ 6.26,7.76] \times 10^{-9}$\\
\hline
\end{tabular}
}

\vspace{-0.3cm}
\caption{Simulation results for the case when each node has a separate buffer}\label{t:sim5}
\end{center}
\vspace{-0.7cm}
\end{table}
There turns out to be only eight effective gradients (out of a maximum of 32).
The results of five consecutive estimations
using the subsolution based IS algorithm are displayed in Table \ref{t:sim5}.
Once again, the estimation results are close to the exact value
$p_{19} = 6.8601 \times 10^{-9}$ and the formal $95\%$ confidence intervals are tight
and happen to contain the exact value.

\section{Discussion}\label{s:discussion}
The goal of the present paper was to extend the IS algorithms in \cite{DSW},
which looked at tandem Jackson networks, to more general networks. We thought tree networks were
an interesting generalization and a comparison with the algorithms
in \cite{DSW}
will reveal that the tree networks require much more sophisticated subsolutions
and IS algorithms for asymptotic optimality.
\cite{yeniDW}  proves a further
generalization 
to arbitrary
stable Jackson networks. In this section we would like to discuss how the
results in \cite{yeniDW} relate to our results.

Let $p_{i,j} = \mu_{i,j}/\mu_i$ denote the routing probability from
node $i$ to $j$, where $j$ is allowed to take the value $0$.
In the notation of the present paper, the IS algorithm in 
\cite{yeniDW} can be described as 
follows. Define the effective rate for the boundary $b$ as:
\begin{equation}\label{e:Mib2}
M_i(b) \doteq \begin{cases}
\mu_i, & \text{ if $b(i)=1$,}\\
\min\left(\mu_i, 
\sum_{k:i\rightarrow k} 
\frac{p_{i,k} \Lambda_i }{\Lambda_k } M_k(b) + \mu'_{i,0}\right), & 
\text{ if $b(i)=0$}.
\end{cases}
\end{equation}
As before if a node is nonempty under $b$, i.e., $b(i)=0$,
then its effective rate is just the service rate $\mu_i$.
If it is empty, one now takes a {\em weighted} sum of the effective rates of its neighbors, as before this sum is min'ed with $\mu_i$.
The weight of $M_k(b)$ is the fraction of the $k^{th}$ node's traffic in the fluid model
that is coming from node $i$. 
This fraction is always $1$ for a tree network and thus for such
networks \eqref{e:Mib2} reduces to \eqref{e:Mib}.
Once the effective rates are defined as above
one proceeds as in subsection \ref{ss:ssubsol}.

We note that \eqref{e:Mib} is a recursive formula: one can start from
the leaves of the network and go up and compute all effective gradients using \eqref{e:Mib}.
In the case of general Jackson networks \eqref{e:Mib2} is an equation that needs to be
solved; as observed in \cite{yeniDW}, it
can be solved by reducing it to a linear equation, which is a generalization  of \eqref{e:defmib}.
It can also be directly solved using \eqref{e:Mib2} itself
and an iterative method.

Another contribution of \cite{yeniDW} is the identification of the
large deviation decay rate $\gamma$ of $p_n$ for any exit 
boundary ${\mathcal S}$
for which such a rate exists. 
In the notation of the present paper,
\cite[Proposition 3.1]{yeniDW} asserts that
$$ \gamma = \inf_{x \in {\mathcal S} } - \langle q, x \rangle$$
where $q$ is the effective or simple gradient of $b=(1,1,1,\dots,1).$
As noted in \cite{yeniDW} this implies that the
IS change of measure given by \eqref{e:Mib2}, or \eqref{e:Mib} for the
case of tree networks, is asymptotically optimal for 
any buffer structure 
${\mathcal S}$ for  which there is a large deviation decay rate.

Finally, we would like to point out a parametrization that seems most natural for \eqref{e:Mib2}.
Define ${\bf M}_i\doteq 1/\rho_i$ and
 ${\bf M}_i(b)\doteq M_i(b)/\Lambda_i$. The first is the ordinary service to arrival ratio of node $i$. 
 The second can be thought of as the effective service to arrival ratio of the same node when the
 system is on boundary $b$.
By convention
let ${\bf M}_0(b) = 1$, i.e., the service to arrival ratio of the outside 
of the system is 
$1$. 
In terms of these new variables \eqref{e:Mib2} is
simply:
\begin{equation}\label{e:Mib3}
	{\bf M}_i(b) \doteq \begin{cases}
		{\bf M}_i, & \text{ if $b(i)=1$}\\
\min\left({\bf M}_i, \sum_{k:i\rightarrow k} p_{i,k} {\bf M}_k(b)\right) , & \text{ if $b(i)=0$},
\end{cases}
\end{equation}
where $k=0$ value is allowed in the summation to denote the outside of the
system.
If  node $i$ is empty, its effective service to arrival ratio is taken to be the average of
the effective ratios of the nodes that are directly connected to $i$. The average is taken
with respect to the routing probabilities. As before the ordinary service to arrival ratio
is an upperbound on the effective one. So if the average exceeds the ordinary, the effective
ratio is set to the ordinary ratio.

The effective gradient $q$ for $b$ will have components $-2\log {\bf M}_i(b)$. And the
change of measure $\bar{p}_b(q)$ is:
\begin{equation*}
\bar{\mu}_{i,j} =\begin{cases} \mu_{i,j}, ~~&\text{if node $i$ is empty}\\
	\mu_{i,j}
	\frac{  {\bf M}_j(b)}{ {\bf M}_i(b) }
	, ~~&\text{if node $i$ is nonempty},
\end{cases}
\end{equation*}
and this is renormalized so that $\bar{\mu}_{i,j}$ sum to $1$. One can 
use \eqref{e:Mib3}
directly to compute the IS algorithm.

\appendix

\section{Proof of Lemma \ref{l:main}}\label{a:proof}
Before we begin, a convention: the decay rate $\gamma$ depends on the 
buffer structure. We used
$\gamma_1$ for the shared buffer (${\mathcal S}_1$) and $\gamma_2$ for the individual buffers
for each node (${\mathcal S}_2$). In the proofs we will simply write $\gamma$.
\begin{lemma}\label{l:simpprop}
Let $q$ be the simple gradient associated with boundary $b$. Then
$H_{\bar{b}}(q) = 0$ for any $\bar{b} \ge b$.
\end{lemma}
\begin{proof}
We first prove that $H_b(q)=0$, or equivalently $N_b(q) = 1$.
Directly from the definitions \eqref{e:Hams}, \eqref{e:defmib} one
sees that $N_b(q)=1$ if and only if
{\small
\begin{equation*}
\sum_{i: b(i) = 1}\left(\sum_{j:i \rightarrow j} m_j(b) + \mu'_{i,0}\right)
 + m_1(b)
 = 
\lambda + \sum_{i: b(i) = 1} \mu_i.
\end{equation*}
}
The definition of $\mu'_{i,0}$ directly imply that 
$\sum_{i=1}^d \mu'_{i,0} = \lambda.$ The above display
follows from this fact and \eqref{e:defmib}.

Next fix a $\bar{b} > b$. 
We will show that $N_{\bar{b}}(q) =1$. 
\begin{equation}\label{e:difNJNI}
N_{\bar{b}}(q) - N_b(q) = 
\sum_{ i: \bar{b}(i)-b(i)=1, i\rightarrow j } 
\hspace{-0.3cm}
\mu_{i,j} e^{\frac{q(i)-q(j)}{2}}
+\sum_{ i:\bar{b}(i)-b(i)=1} \mu_{i,0} e^{q(i)/2} 
- 
\hspace{-0.3cm}
\sum_{ i :\bar{b}(i)-b(i)=1}
\mu_i 
\end{equation}
Fix $i$ such that $\bar{b}(i)-b(i)=1$
and let $C$ denote the terms 
contributed by
the index $i$
in the first two sums. Our goal is now to show that $C=\mu_i$.
This will imply that first two sums and the last sum in \eqref{e:difNJNI}
cancel each other and that $N_{\bar{b}}(q) = N_{b}(q)$.
Because $b(i) = 0$ we have that
\begin{equation}\label{e:mirepp1}
m_i(b) = \sum_{j: i\rightarrow j} m_j(b) + \mu'_{i,0}.
\end{equation}
Then
$$
C=
\mu_{i,0} e^{q(i)/2} + 
\sum_{j: i \rightarrow j} \mu_{i,j} e^{\frac{q(i)-q(j)}{2}}
= \mu_{i,0} \frac{\Lambda_i}{m_i(b)} +  \sum_{j: i \rightarrow j} \mu_{i,j} 
\frac{\Lambda_i}{m_i(b)}\frac{m_j(b)}{\Lambda_j}
$$
At this point the facts 
$\Lambda_j = \Lambda_i \frac{\mu_{i,j}}{\mu_i}$ 
and
$\frac{\mu_{i,0}\Lambda_i}{\mu_i}=\mu'_{i,0}$
and
 \eqref{e:mirepp1}
 and simple arithmetic yield $C=\mu_i$.
Thus the difference in \eqref{e:difNJNI} is zero, i.e.,
$N_{\bar{b}}(q) = N_b(q) = 1$. This finishes the proof of this lemma.
\end{proof}

\begin{lemma}\label{l:egcomp}
Let $q$ be the effective gradient associated with boundary $b$. Then
$ H_{b'}(q) \ge 0$
for all $b' \ge b$.
\end{lemma}
\begin{proof}
$ H_{b'}(q) \ge 0$ if and only if $
N_{b'}(q) \le 1$.
By Lemma \ref{l:effsimp} there exists $\bar{b} \ge b$ such that
$q$ is the simple gradient associated with $\bar{b}$. Then by Lemma
\ref{l:simpprop} 
$
N_{b'}(q) =1
$
for all $b' \ge \bar{b}$. Now take any $b'$ such that $b' < \bar{b}$ and
$b' \ge b$. 
 Because $\bar{b} > b' \ge b$ we have
 {\small
\begin{align}\label{e:tmp0l5}\notag
&N_{\bar{b}}(q) - N_b(q)  \\
&~~~=\sum_{ i: \bar{b}(i)-b(i) = 1}\left( \sum_{j:i\rightarrow j } 
\mu_{i,j} e^{\frac{q(i)-q(j)}{2}}
+ \mu_{i,0} e^{q(i)/2} \right)
- 
\hspace{-0.3cm}
\sum_{ i: \bar{b}(i)-b(i) = 1}
\mu_i \notag
\\
&~~~=
\sum_{ i: \bar{b}(i)-b(i) = 1}
\left( \sum_{j:i\rightarrow j } 
\mu_{i,j}\frac{\Lambda_i}{M_i(b)}\frac{M_j(b)}{\Lambda_j}
+\mu_{i,0}\frac{\Lambda_i}{M_i(b)} \right)
- 
\hspace{-0.3cm}
\sum_{ i: \bar{b}(i)-b(i) = 1}
\mu_i\notag\\
&~~~=
\sum_{ i: \bar{b}(i)-b(i) = 1}
\left( 
\mu_i \frac{\sum_{j:i\rightarrow j } M_j(b) + \mu'_{i,0}}{M_i(b)}\right)
- 
\hspace{-0.3cm}
\sum_{ i: \bar{b}(i)-b(i) = 1}
\mu_i
\end{align}
}
Now by the construction of $\bar{b}$, $\bar{b}(i)-b(i)=0$ if and only
if
$
M_i(b) = \mu_i \le \sum_{i\rightarrow j} M_j(b) + \mu'_{i,0},
$
The last display and \eqref{e:tmp0l5} imply
$N_{\bar{b}}(q)  \ge N_b(q).$
Because $N_b(q) =1$ (because $q$ is the simple gradient associated with
boundary $b$) this finishes the proof of this lemma.
\end{proof}


\begin{proof}[Proof of Lemma \ref{l:main}]
The proof is this lemma is similar to the proof of Theorem 4.31 in \cite{thesis}.
For small positive real numbers $\delta,\epsilon$ let $W^{\epsilon,\delta}$
be defined as in \eqref{e:Wep}. 
For ease of notation we will drop the superscript $(\epsilon,\delta)$ and
write $W$.
We would like to prove the following:
there is a constant $C_1$ that only depends on the parameter system
such that
for all $x \in {\mathbb R}_+^d$ 
$
H_{b}(DW(x)) \ge -C_1 \exp(-\epsilon/\delta),
$
where $b$ defined in \eqref{e:bx} is the boundary corresponding to $x$.
Let $E$ be the
set of effective gradients $q$
such that there is a boundary $b' \le b$ with effective gradient $q$. 
Define
$
q' = \sum_{q_l \in E} w_l^{\epsilon,\delta}(x) q_l,
$
where $w_l^{\epsilon,\delta}$ are the weights defined in \eqref{e:wi}.
Once again to ease notation, we drop the superscript $(\epsilon,\delta)$.
Its definition directly implies that $H_b$ is concave and Lipschitz continuous.
By Lemma \ref{l:egcomp} we have that
$H_{b}(q) \ge 0$ for $q \in E$. 
This fact and the concavity of $H_{b}$ and
$H_{b}(0)=0$
imply that
$H_{b}(q') \ge 0.$
This, \eqref{e:wi} and the Lipschitz continuity
of $H_b$ give
\begin{align*}
H_b(DW(x)) &= H_b(q') + H_b(DW(x)) - H_b(q') \ge |H_b(DW(x)) - H_b(q')|\\
           &\ge K | q' - DW(x)|   =  -K\sum_{q^l \in E^c } w_l(x) | q^l|.
\end{align*}
The last inequality follows from \eqref{e:wi} and the triangle inequality.
Therefore to prove the first part of Lemma \ref{l:main} it is enough
to prove
$
w_l(x)\le \exp(-\epsilon/\delta),
$
for $l$ such that $q_l \in E^c$. 

By its definition \eqref{e:wi} $w_l$ equals
{\small
\begin{align}
\label{e:boundonw}
w_l(x) = \frac{\exp\left\{-{W}_l^\epsilon(x)/\delta\right\}}
{\sum_{j=1}^L \exp\left\{- {W}_j^\epsilon(x)/\delta \right\}}
&= 
\frac{ \exp\left\{(\alpha_l\epsilon -  \langle q_l, x \rangle)/\delta\right\}}
{\sum_{j=1}^L \exp\left\{(\alpha_j\epsilon -  \langle q_j, x \rangle)/\delta \right\}}\notag\\
&\le
\frac{\exp\left\{(\alpha_l\epsilon -  \langle q_l, x \rangle)/\delta\right\}}
{\exp\left\{(\alpha_{j_0}\epsilon -  \langle q_{j_0}, x \rangle)/\delta 
\right\}},
\end{align}
}
where $q_{j_0}$ is an effective gradient to be selected.
By Definition \ref{d:defa}, $\alpha_l$ is one plus the number of $0$'s in the
the boundary (bitmap) $r$ whose simple gradient equals
$q_l$.
Form the bitmap $\tilde{r}$
from $r$ as follows: if
$r(i) = 1$ but $b_x(i) = 0$ then set $\tilde{r}(i) = 0$ 
otherwise set $\tilde{r}(i) = r(i)$.
By this construction
$\tilde{r} \le b_x$ and $\tilde{r} < r $. The last inequality is strict,
because otherwise we would have $b_x= r$ which would imply, 
by Lemma \ref{l:egcomp}, $H_{b_x}(q_l) \ge 0$ which in turn contradicts
$q_l \notin E$.  Let $q_{j_0}$ 
be the effective gradient associated
with the bitmap $\tilde{r}$. $\tilde{r} \le b_x$ and Lemma 
\ref{l:egcomp} imply 
that $H_{b_x}(q_{j_0}) \ge 0$. This implies that $q_{j_0} \in E$ and
consequently $q_{j_0} \neq q_l \in E^c$. 
These facts and the strict inequality 
$\tilde{r} < r$ imply that $\alpha_{j_0} - \alpha_l \ge 1$.

Furthermore, remember $x$ is such that $x_i = 0$ if $b_x(i) = 0$. The
bitmaps $r$ and $\tilde{r}$ differ only at such $i$. Then the effective
gradients of these bitmaps, namely $q_l$ and $q_{j_0}$ will also differ
only at such $i$. This means $\langle q_l , x \rangle = \langle q_{j_0}, x\rangle.$
These considerations and \eqref{e:boundonw} imply
$
w_l(x) \le  \exp(-\epsilon/\delta)
$
and hence the first part of Lemma \ref{l:main}.

By its definition
$$
W(0) =
-\delta 
\log\sum_{l=1}^L \exp\left\{-\frac{2\gamma -\alpha_l \epsilon}{\delta} \right\}
= 2\gamma -\epsilon\left(\frac{\delta}{\epsilon}
\log\sum_{l=1}^L \exp\left\{
\frac{\alpha_l}{\delta/\epsilon} \right\}\right)
$$
This proves the second part of Lemma 
\ref{l:main}.

Now let us prove the third part. Let $q_L$ be the effective gradient
of the boundary $1=(1,1,1,\dots,1,1)$.
For $x \in {\mathbb R}^d_+$ with
$x_1+x_2+\cdots+x_d =1 $ we have the following estimate:
$$W(x) = -\delta 
\log\sum_{l=1}^L \exp\left\{-\frac{1}{\delta}(2\gamma  -\alpha_l\epsilon + \langle q_l, x \rangle)\right\}
\le  \langle q_L, x \rangle +2\gamma - \alpha_L\epsilon.
$$
By definition $q_L(i) = 2\log \frac{\mu_i}{\Lambda_i}$. This and \eqref{e:ldres} imply
that the last line is less than $ -\alpha_L \epsilon.$
This finishes the proof of the third part of Lemma \ref{l:main}.
It only remains to prove the last part. Differentiating the
first expression in \eqref{e:wi} gives:
$
\frac{\partial^2 W}{\partial x_j \partial x_i}(x) = \sum_{l=1}^L \frac{\partial w_l}{\partial x_j}(x) q_l(i).
$
Differentiating the second expression in \eqref{e:wi} gives:
$
\frac{\partial w_l}{\partial x_j}(x) = \frac{1}{\delta} w_l(x) \left(\sum_{k=1}^L w_k(x)( q_k(j) -q_l(j))\right).
$  
These imply the bound in part 4 of Lemma \ref{l:main}, which is what
we wanted to prove.
\end{proof}
\bibliography{../../../bibliography/IS} 
\end{document}